\documentclass [12pt]{article}
\usepackage{amsmath,amssymb}
\usepackage{amssymb,verbatim}

\usepackage{enumerate}

\newcommand{\interior}[1]{%
  {\kern0pt#1}^{\mathrm{o}}%
}
\usepackage{mathtools}
\usepackage{graphicx}
\newcommand*{\medcap}{\mathbin{\scalebox{1.5}{\ensuremath{\cap}}}}%
\usepackage{scalerel}
\usepackage{relsize}
\usepackage[T1]{fontenc}
\usepackage[tableposition=top]{caption}
\usepackage[countmax]{subfloat}

\usepackage{caption} 
\usepackage[nospace,noadjust]{cite}

\usepackage{float}
\newcommand\preceqdot{\mathrel{\ooalign{$\preceq$\cr
  \hidewidth\raise0.125ex\hbox{$\cdot\mkern0.5mu$}\cr}}}
\newcommand\precdott{\mathrel{\ooalign{$\prec$\cr
  \hidewidth\raise0.015ex\hbox{$\cdot\mkern0.5mu$}\cr}}} 
\usepackage{mathabx}
\usepackage{float}
\usepackage{blindtext}
\usepackage{titlesec}
\title{Sections and Chapters}
\usepackage{enumitem}
\usepackage{multirow}
\usepackage{array}
\usepackage[colorlinks]{hyperref}
\newcolumntype{L}[1]{>{\raggedright\let\newline\\\arraybackslash\hspace{0pt}}m{#1}}
\newcolumntype{C}[1]{>{\centering\let\newline\\\arraybackslash\hspace{0pt}}m{#1}}
\newcolumntype{R}[1]{>{\raggedleft\let\newline\\\arraybackslash\hspace{0pt}}m{#1}}
\usepackage{blkarray}
\usepackage{amsmath}
\usepackage{mathrsfs}
\usepackage{amsfonts}
\usepackage{tikz-cd} 
\usepackage{subfig}
\usepackage{faktor}
\usepackage{xfrac}
\usepackage{faktor} 
\usepackage{amsmath} 
\usepackage{amssymb}
\usepackage{mathtools}

\usepackage[utf8]{inputenc}
\usepackage{graphicx}
\usepackage{graphicx,amssymb,amstext,amsmath}
\usepackage{tikz}
\usepackage{amsmath}
\usepackage{amssymb}
\usepackage{amsthm}
\usepackage{hebcal}
\usetikzlibrary{arrows, decorations.markings}
\usepackage{blindtext}
\usepackage{amsthm}
\usepackage[T1]{fontenc}
\usepackage[utf8]{inputenc}
\usepackage{blindtext}
\usepackage[T1]{fontenc}
\usepackage[utf8]{inputenc}
\newcommand\Item[1][]{%
  \ifx\relax#1\relax  \item \else \item[#1] \fi
  \abovedisplayskip=0pt\abovedisplayshortskip=0pt~\vspace*{-\baselineskip}}
\usepackage{tikz}
\usetikzlibrary{positioning, arrows.meta}

\usepackage{cite}
\newtheorem{thm}{Theorem}[section]
\newtheorem{lemma}[thm]{Lemma}
\newtheorem{prop}[thm]{Proposition}

\newtheorem{innercustomgeneric}{\customgenericname}
\providecommand{\customgenericname}{}
\newcommand{\newcustomtheorem}[2]{%
  \newenvironment{#1}[1]
  {%
   \renewcommand\customgenericname{#2}%
   \renewcommand\theinnercustomgeneric{##1}%
   \innercustomgeneric
  }
  {\endinnercustomgeneric}
}

\newcustomtheorem{customthm}{Theorem}
\newcustomtheorem{customlemma}{Lemma}

\newtheorem{remark}[thm]{Remark}

\newtheorem{question
}{Question}

\def\0{{\bf 0}}

\def\N{{\bf N}}

\def\Q{{\bf Q}}
\def\R{{\bf R}}
\def\Z{{\bf Z}}

\def\vol{{\rm vol}}

\def\vis{{\rm vis}}
\def\SL{{\rm SL}}
\def\ASL{{\rm ASL}}

\usepackage{subfig}

\newcommand{\df}{{\, \stackrel{\mathrm{def}}{=}\, }}
\newcommand{\covol}{\mathrm{covol}}

\makeatletter
\def\keywords{\xdef\@thefnmark{}\@footnotetext}
\title {On the error bounds for visible points in some cut-and-project sets}
\date{}

\author{
Ilya Gringlaz\footnote{Department of Mathematics, Tel Aviv University, Israel. E-mail: redlid@gmail.com}
\and 
Rishi Kumar\footnote{Department of Mathematics, Tel Aviv University, Israel. E-mail: rkumar@tauex.tau.ac.il}
\and
Barak Weiss\footnote{Department of Mathematics, Tel Aviv University, Israel. E-mail: barakw@tauex.tau.ac.il}}
\begin{document}
\maketitle
\begin{abstract}
We study points in cut-and-project sets which are visible from the
origin, continuing a direction of inquiry initiated in \cite{Marklof3,
  Gustav}, where the asymptotic density of visible points was
investigated. We establish an error bound
for the density of visible 
points in certain cases. We also prove that the set of visible points
in irreducible 
cut-and-project sets with star-shaped windows is never relatively dense.
\end{abstract}

\section{Introduction}
Let $\mathcal{P}\subset \R^d$ be a locally finite point set, let $D
\subset \R^d$ be a bounded
measurable set with $\vol(D)>0$, and for $T>0$ let $TD$ denote the
dilated set $\{tx: x \in D, t \in [0, T]\}$. 
The {\em asymptotic density of $\mathcal{P}$ with respect to $D$} is
defined to be 
\begin{equation}\label{eq: def density}
 \theta (\mathcal{P}) \df  \lim_{T\to \infty}\frac{\#(\mathcal{P}\medcap
   TD)}{\textup{vol}(TD)},    
\end{equation}
provided the limit exists. In general, the existence of the limit and its value may depend on the
choice of the averaging set  $D$, but this will not play a role in this paper and we
suppress the dependence of $\theta$ on $D$ from the notation. The most
commonly studied case is 
the case 
in which $D$ is the unit ball with respect to some norm on $\R^d$. 
We denote
$$\mathcal{P}_\star \df \mathcal{P} \smallsetminus \{0\}, \ \ \
\mathcal{P}_{\vis}\df  \{y\in \mathcal{P}_\star\, : \, t
y\not\in 
\mathcal{P}, \, \forall \, t\in (0,1)\},$$ 
the set of nonzero points and the set of points of $\mathcal{P}$ which
are  visible from the origin. For 
certain 
sets $\mathcal{P} \subset \R^d,
 d \geq 2$ for which the asymptotic density of
$\mathcal{P}_{\vis}$ has recently been established, we will be
interested in the rate of
convergence of the limit in \eqref{eq: 
  def density}. We will also be interested in the question of relative
density of the set $\mathcal{P}_{\vis}$.

To set the stage we review
some of 
what is known  in case $\mathcal{P}= \Z^d$;
i.e., the lattice of integer points. In this case, the set of visible
points from the origin is given by the primitive lattice points: 
$$\Z_{\vis}^d= \left\{(n_1,\ldots, n_d)\in \Z^d_\star\, : \,
  \gcd(n_1,\ldots,n_d)=1 \right \}.$$
When $D$ is the unit ball of a norm, an elementary and classical
argument using M\"obius inversion shows that $\theta(\Z_{\vis}^d)=
1/\zeta(d)$ (see \cite{Moody}, 
also see the appendix of \cite{Moody} for a discussion of 
more general averaging sets). Here $\zeta(s) = \sum_{n \in \N} n^{-s}$ denotes
the Riemann zeta function. The question of error estimates in this 
convergence has been extensively studied and we give a sample of results. For the unit
ball with respect to the $\ell_\infty$ norm we have (see \cite{LIU, Nym})
\begin{equation*}
\begin{split}
\frac{\#\left(\Z^d_{\vis}\medcap [-T,T]^d\right)}{\vol([-T,T]^d)}&= \frac{1}{\zeta(d)} + \begin{cases}
O\left(\frac{(\log T)^{2/3}(\log \log T)^{1/3}}{T}\right),&\qquad d=2,\\
O\left(\frac{1}{T}\right),&\qquad d\geq 3.
\end{cases}     
\end{split}   
\end{equation*}
See \cite{Schmidt} for the case of
Euclidean balls, \cite{Huxley_Nowak, BNRW, Horesh_Karasik} for a
discussion of more averaging sets. In this example $\mathcal{P} = \Z^d$, the
density $\theta$ exists and is independent of $D$, for a large variety
of sets $D$, but the error term depends quite delicately on $D$.


A subset $\mathcal{P}\subset \R^d$ is called {\em
  relatively dense} if  there exist a constant $R>0$ such that
for
every $x\in \R^d$, we have $d(x, \mathcal{P})\leq R$. If $\mathcal{P}$ is not relatively
dense, we say that it {\em has arbitrarily large holes}; that is, for
any $R>0$ there is a ball $B$ of radius $R$ for which $B \medcap \mathcal{P}=
\varnothing.$  
It was shown in \cite{Erdos1} that the set $\Z^d_{\vis}$ has
arbitrarily large holes (see also \cite[Prop. 4]{Moody} and the proof of Lemma
\ref{lemma on invisible points2} below).




In this paper we will consider {\em cut-and-project sets} (also
referred to as model sets), which are
defined as follows (see \cite{Meyer, Baake1, Marklof2}). A {\em grid}
in $\R^n$ is the image of a lattice under a translation, that is a set
of the form $\mathbf{y} + g \Z^n$ where $\mathbf{y} \in \R^n$ and $g
\in \mathrm{GL}(n,\R)$.  
Let $n = d+m$ for positive integers $n,d,m$, and let
$\pi_{\textup{phys}}$ and $\pi_{\textup{int}}$ denote the 
projections of $\R^n$ onto the two summands in the direct sum
decomposition
\begin{equation}\label{eq: direct sum decomposition}
  \R^n= \R^d\oplus \R^m.
  \end{equation}
The first and second summands in this 
decomposition are referred to as {\em physical} and {\em internal}
space respectively, and we will continue to denote them by $\R^d,
\R^m$ in the remainder of the paper. Let $\mathcal{L}\subset \R^n$ be a grid and
$\mathcal{W}\subset 
\R^m$ be a subset referred to as the {\it window}. In this paper we
will always assume that $\mathcal{W}$ has non-empty
interior, and is Jordan measurable (that is,  bounded and with
boundary of  zero
Lebesgue measure). We will impose certain additional conditions on 
$\mathcal{W}$ further below. The
cut-and-project set associated with $(\mathcal{W}, \mathcal{L})$ is
defined as 
$$\Lambda(\mathcal{W},\mathcal{L})= \{\pi_{\textup{phys}}(y)\, : \,
y\in \mathcal{L}, \, \pi_{\textup{int}}(y)\in \mathcal{W}\}\subset
\R^d.$$
We say that $\Lambda = \Lambda ( \mathcal{W}, \mathcal{L})$ is {\em irreducible}
if $\pi_{\textup{int}}(\mathcal{L})$ is dense in $\R^m$ and
$\pi_{\textup{phys}}|_{\mathcal{L}}$ is injective. 
In this case $\Lambda$ is relatively dense,
and its 
 density $\theta(\Lambda)$ exists (whenever $D$ is 
 Jordan-measurable) and
is given by 
$$\theta\left(\Lambda(\mathcal{W},\mathcal{L})\right)=
\frac{\vol(\mathcal{W})}{\covol(\mathcal{L})}$$ 
(see \cite{Hof, Schlotmann, Marklof2}). Here $\covol(\mathcal{L})$
denotes the {\em covolume} of $\mathcal{L}$, defined as
$\covol(\mathbf{y} + g \Z^n) = |\det g |.$ See \cite{Yaar1, Alan,
  TrevinoSchmieding, Barak1} 
for some
results about the rate of convergence in \eqref{eq: def density} for
cut-and-project sets.

Recently, 
Marklof and  Str{\"o}mbergsson \cite[Theorem 1] {Marklof3} proved that
for any irreducible cut-and-project set
$\Lambda$, the density of visible points
exists, is the same for all Jordan measurable $D$, and satisfies 
\begin{equation}\label{eq: MS inequalities}
  0< \theta(\Lambda_{\vis})\leq
  \theta(\Lambda).
  \end{equation}
They also observed (\cite[p.2] {Marklof3}) that for `generic choices'
(see Proposition \ref{prop: MS vindicated 1})
of the grid $\mathcal{L}$ one has equality in the right-hand side of
\eqref{eq: MS inequalities}. On 
the other hand Hammarhjelm \cite{Gustav} gave certain examples of
cut-and-project sets in $\R^2$ for which the latter inequality is
strict, and computed  
$\theta(\Lambda_{\vis})$ explicitly (some of these examples had also been
considered  by Sing \cite{Sing_slides}). We will refer to the cases
considered by Hammarhjelm as {\em Hammarhjelm examples} and
give a definition in \S \ref{subsec: Hammarhjelm examples}.

In this paper we prove the following results. In all of these results,
$\Lambda = \Lambda(\mathcal{W}, \mathcal{L})$ is an irreducible
cut-and-project set in $\R^d, \, d \geq 2$,  and the 
averaging sets $D$ are Jordan measurable. Recall
that $\mathcal{W}$ is {\em star-shaped
with respect to the origin} if for any $w \in \mathcal{W}$, the
segment $\{tw: t \in [0,1]\}$ contained in $\mathcal{W}$.

\begin{thm}\label{thm: large holes}
Assume $\mathcal{W}$ is star-shaped with respect to the origin, and
$\mathcal{L}$ is a lattice. Then $\Lambda_{\vis}$
has arbitrarily large holes. 
\end{thm}

The special case of the set of visible points of the Amman-Beenker
seems to have 
known before, see \cite{Sing_slides} and \cite[p. 427]{Baake1}.

Let $m$ denote the Haar measure on the group
$\mathrm{GL}(n,\R)$.

\begin{thm}\label{thm: generic density}
Assume $\mathcal{W}$ is star-shaped with respect to the origin.
Then for
$m$-a.e.\,$g \in \mathrm{GL}(n,\R)$, the 
lattice $\mathcal{L} = g \Z^n$ satisfies 
$\theta(\Lambda_{\vis}
)  = \frac{1}{\zeta(n)}
\theta(\Lambda
)$ for any
averaging set. 
Moreover for any averaging set $D \subset \R^d$, for any $\varepsilon>0$, 
$m$-a.e.\,$g \in \mathrm{GL}(n,\R)$, we have an error bound
$$
\left|\frac{ \# \left(\Lambda_\vis \medcap TD\right)}{\vol(TD)} -
  \theta(\Lambda_\vis)\right| = O\left(\vol(TD)^{-\frac{1}{3}+\varepsilon}\right). 
$$

\end{thm}
The main ingredient in the proof of Theorem \ref{thm: generic density} is work of
Fairchild and Han \cite{Han1} which will be recalled below. We stress
that Theorem \ref{thm: generic density} does not contradict the
observation of \cite{Marklof3} as we deal with generic {\em lattices},
not generic {\em grids}; cf. Propositions \ref{prop: MS vindicated 1}
and \ref{prop: ae mu}. 
Put differently, our definition of visibility concerns {\em visibility
  from the origin} but if one replaces it with {\em visibility from a
  fixed basepoint in $\Lambda$} one would see that typically, strict
inequality holds in \eqref{eq: MS inequalities}.

\begin{thm}\label{thm: Gustav examples}
For the Hammarhjelm examples, if the averaging set $D$ is convex, for
any $\varepsilon>0$ we
have an error bound 
$$
\left|\frac{ \# \left(\Lambda_\vis \medcap TD\right)}{\vol(TD)} -
  \theta(\Lambda_\vis)\right| = O\left(\vol(TD)^{-\frac{1}{4}+\varepsilon}\right). 
$$
  \end{thm}


  \subsection{Acknowledgements}
The authors are grateful to Anish Ghosh for directing their attention
to \cite{Han1}, and to Lior Bary-Soroker and Yaar Solomon for helpful
discussions. 
  The authors' work is supported by the  grants ISF 2021/24, ISF-NCSF 3739/21 and 
ISF 2860/24.

\section{Visible points in a lattice $\mathcal{L}$ and 
  in $\Lambda(\mathcal{W}, \mathcal{L})$. }
Denote
$$\Lambda(\mathcal{W},\mathcal{L}_{\vis}) \df \{\pi_{\textup{phys}}(y)\, : \, y\in
\mathcal{L}_{\vis}\,, \pi_{\textup{int}}(y) \in \mathcal{W}\}.$$ 
The following lemma provides a relation between the visible points of
a lattice and the visible points of cut-and-project sets. 
\begin{lemma}\label{lemma on the relation of visible points}
Let $\mathcal{L}$ be a lattice in $\R^n$, and  $\mathcal{W}$ be a
window which is star-shaped with respect to the origin.  Then
\begin{equation}\label{eq: inclusion lemma}
  \Lambda(\mathcal{W},\mathcal{L})_{\vis} \subset 
  \Lambda(\mathcal{W},\mathcal{L}_{\vis}). 
  \end{equation}
\end{lemma}
\begin{proof}
Let $x\in \Lambda(\mathcal{W},\mathcal{L})_\star \smallsetminus
\Lambda(\mathcal{W},\mathcal{L}_{\vis})$. By definition, there exists
$y\in \mathcal{L}_\star\smallsetminus 
\mathcal{L}_{\vis}$ such that $x=\pi_{\textup{phys}}(y)$. This implies
that there exists a real number $t \in (0,1)$, such that $y'\df ty\in
\mathcal{L}$. Since $\mathcal{W}$ is star-shaped with respect to the
origin and $x\in
\Lambda(\mathcal{W},\mathcal{L})$, we obtain 
$$\pi_{\textup{int}}(y')= t\pi_{\textup{int}}(y) \in \mathcal{W}.$$
Consequently,
$$tx = t \pi_{\textup{phys}}(y)= \pi_{\textup{phys}}\left(t y\right)= \pi_{\mathrm{phys}}(y')
\in \Lambda(\mathcal{W},\mathcal{L}).$$ 
This shows that $x\notin \Lambda(\mathcal{W},\mathcal{L})_{\vis}$.
\end{proof}

In \cite{Gustav}, Hammarhjelm gave examples for which the inclusion in
\eqref{eq: inclusion lemma} is strict. The following result gives a
condition guaranteeing equality in \eqref{eq: inclusion lemma}.

\begin{prop}\label{prop: condition for equality}
Let $\Lambda = \Lambda(\mathcal{W}, \mathcal{L})$, where $\mathcal{L}$
is a lattice  and $\mathcal{W}$ is star-shaped with respect to the
 origin.
If we have strict inclusion in
\eqref{eq: inclusion lemma} 
then there are linearly independent $y_1,
y_2 \in \mathcal{L}$ such that 
$ \dim \left(\textup{span}_{\R}(y_1, y_2)\medcap \R^m \right) \geq 1$. 
  \end{prop}

  \begin{proof}
Suppose that $\Lambda(\mathcal{W}, \mathcal{L}
_{\vis})\smallsetminus \Lambda(\mathcal{W}, \mathcal{L})_{\vis}
 \neq \varnothing$. Then there exists  $y_1\in \mathcal{L}_{\vis}$ such
that $x\df \pi_{\textup{phys}}(y_1) \in \Lambda(\mathcal{W},\mathcal{L}_\vis)$
and $x\not\in \Lambda(\mathcal{W},\mathcal{L})_{\vis}$. If $x =0$ then
$y_1 \in \R^m$ and the desired conclusion holds, with $y_2$ any
element of $\mathcal{L}$ which is not a scalar multiple of $y_1$. Thus we can assume that
$x \neq 0$ and thus 
there exist $y_2 \in\mathcal{L}$ and $t \in(0,1)$ such that 
$$x'\df \pi_{\textup{phys}}(y_2)= t \pi_{\textup{phys}}(y_1)= t x.$$ 
Clearly $y_1$ and $y_2$ are nonzero. If there was some $c \in \R$ such
that $y_2 = cy_1$, then by applying $\pi_{\mathrm{phys}}$ we find that
we must have $c = t \in (0,1)$, in contradiction to the fact that 
$y_1 \in \mathcal{L}_\vis$. Therefore $y_1, y_2$ are linearly
independent; i.e.,  
$$\dim U =2, \text{ where } U \df \textup{span}_{\R}(y_1, y_2).$$ 
We have $\dim(\pi_{\textup{phys}}(U))=1$, and since the kernel of
$\pi_{\textup{phys}}$ is the space $\R^m$, this implies 
$\dim(U\medcap \R^m)=1$.
    \end{proof}

We say that {\em  $\mathcal{H}$ is a hole in $\mathcal{P}$} if $\mathcal{H} \medcap
\mathcal{P}= \varnothing.$
The following statement follows easily from Lemma \ref{lemma on the
  relation of visible points}.

\begin{lemma}\label{lemma on invisible points1} Let $\mathcal{L}$ be a
  lattice in $\R^n$, let $\mathcal{W}$ be a window which is
  star-shaped around the origin, and let $\Lambda =
  \Lambda(\mathcal{W}, \mathcal{L})$. Suppose that $\mathcal{H}$ is a hole in
  $\mathcal{L}_{\vis}$ and $\mathcal{H}_1 \subset \R^d$ satisfies
  $\mathcal{H}_1\times \mathcal{W}\subset 
  \mathcal{H}$. Then $\mathcal{H}_1$
  is a  hole of $\Lambda(\mathcal{W},\mathcal{L})_{\vis}$.   
\end{lemma}

Let $\mathcal{P}\subset \R^n$ and let $V \subset \R^n$ be a
subset. We say that $\mathcal{P}$ {\em contains arbitrarily large
  holes along $V$} if for any $R>0$ there is a ball $B$ of radius $R$
and 
with center in $V$ such that $B \medcap \mathcal{P}= \varnothing.$

\begin{lemma}\label{lemma on invisible points2}
Let $V \subset \R^n$ be a linear subspace such that $V + \Z^n$ is dense in
$\R^n$. Then $\Z^n_\vis$ has arbitrarily large holes along $V$. 
\end{lemma}

\begin{proof}
Since $V + \Z^n$ is dense in
$\Z^n$ we obtain by rescaling that for any $N>0$,
\begin{equation}\label{eq: dense inclusion}
  V+
N \Z^n = 
N \cdot  (V +
\Z^n )  \subset \R^n 
\end{equation}
is a dense inclusion.

Let $\mathcal{U}$ be an open bounded subset of $\R^n$, and for each $A \in \N$ and each
$\mathbf{x}  = (k_1, \ldots, k_n)\in
\Z^n$, let
$$ C(A,\mathbf{x}) \df \Z^n \medcap \left(\mathbf{x} + [-A, A]^n\right) =
\left\{(k_1 + i_1, \ldots, k_n + i_n) \,: \, |i_j| \leq A, \, j= 1, 
\ldots, n \right\}.
$$
We claim that there are $N>0$ and  $\mathbf{x}_0 \in \Z^n
$ such that for each $\mathbf{x} \in \mathbf{x}_0 + N
\Z^n$ we have 
$C (A, \mathbf{x}) \medcap \Z^n_\vis = \varnothing.$
To see this, let $\mathcal{I} \df \Z^n \medcap [-A, A]^n$, and
for each tuple $(i_1,\ldots, i_n) \in \mathcal{I}$, choose a prime
$P_{i_1,\ldots, i_n}$ so that these primes are all distinct.
Then, by the Chinese remainder theorem, for each fixed $j \in \{1,
\ldots, n\}$  there exists a
solution
$k_j \in \Z$ to the system of congruences
\begin{equation}\label{eq: CRT}
   k_j \equiv -i_j \ (\textup{mod} \ P_{i_1,\ldots, i_n}), \qquad
   (i_1, \ldots, i_n) \in \mathcal{I}.
 \end{equation}
Define
$$\mathbf{x}_0= (k_1,\ldots,k_n)\in \Z^n  \ \ \text{ and } \ \ N\df
\prod_{(i_1,\ldots, i_n)\in \mathcal{I}} P_{i_i,\ldots, i_n}.$$ 
Then \eqref{eq: CRT} remains true if $(k_1, \ldots, k_n)$ are the
coordinates of any $\mathbf{x} \in \mathbf{x}_0 + N\Z^n$, and the
choice \eqref{eq: CRT} ensures that for such vectors, $\gcd(k_1 + 
i_1, \ldots, k_n + i_n) \geq P_{i_1, \ldots, i_n}.$ 
This proves the claim.

Now given $R>0$, let $A$ be large enough so that for any
$\mathbf{x} \in \mathcal{U} + V$, the set $C(A, \mathbf{x})$ contains
the ball $B(\mathbf{x}_1, R)$ for some $\mathbf{x}_1 \in V$. Such an
$A$ exists because $\mathcal{U}$ is bounded. 
Since the inclusion in \eqref{eq: dense inclusion} is dense, we have
$\mathcal{U} + V+ N\Z^n = \R^n,$
and thus
$$\mathbf{x}_0 = u + v - \mathbf{x}, \ \ \text{ where } u \in
\mathcal{U}, \ v \in V, \ -\mathbf{x} \in N\Z^n.$$
In particular 
$$\mathbf{x}_0+ \mathbf{x} \in (\mathbf{x}_0+ N\Z^n )\medcap (\mathcal{U} + V).$$
Now using the claim, we see that
$\Z^n_\vis$ contains arbitrarily large holes along $V$. 
\end{proof}



\begin{proof}[Proof of Theorem \ref{thm: large holes}]
Write $\Lambda = \Lambda(\mathcal{W}, \mathcal{L})$, where
$\mathcal{W}$ is star-shaped with respect to the origin, $\mathcal{L}$
is a lattice, and $\pi_{\textup{int}}(\mathcal{L})$  is a dense
subset of $\R^m$. By Lemma \ref{lemma on invisible points1}, it suffices to show that
$\mathcal{L}_\vis$ has arbitrarily large holes along $\R^d$.

To this
end, write $\mathcal{L}  = g \Z^n$ for $g \in \mathrm{GL}(n, \R)$ and let
$V \df g^{-1} \R^d.$ Since $\pi_{\textup{int}}(\mathcal{L})$  is a dense
subset of $\R^m$, and the standard topolgy on $\R^m$ is the quotient
topology for the projection $\pi_{\textup{int}}: \R^n \to \R^m$, we
see that $\R^d + \mathcal{L}$ is dense in $\R^n$. Since linear
transformations are homeomorphisms, this implies that
$V+ \Z^n$ is dense in $\R^n$. Applying Lemma
\ref{lemma on invisible points2}, we see that $\Z^n_\vis$ has arbitrarily large 
holes along $V$. 
Since linear transformations are
bi-Lipschitz maps, and send visible points to visible points, this
implies that $\mathcal{L}_\vis$ has arbitrarily large holes along
$\R^d$, as required.
  \end{proof}


\section{Random cut-and-project sets and point counting}
In this section we discuss probability measures on discrete subsets of
$\R^d$, explain what we mean by a `random' cut-and-project
set, and prove Theorem \ref{thm: generic density}. In fact we will
prove the stronger Theorem \ref{theorem on random cut-and-project sets} below. 
For a complete metric space $X$, we denote by $\mathrm{Cl}(X)$ the space of
closed subsets of $X$, equipped with the Chabauty-Fell topology (see
\cite[\S 2.2]{Barak1} and references therein).

\begin{prop}\label{prop: MS vindicated}
Let $\nu$ be any measure on $\mathrm{Cl}(\R^n)$ which is supported on discrete
countable sets and invariant under translations. Then for
$\nu$-a.e. $\mathcal{P}$, $\mathcal{P}_\star  = 
\mathcal{P}_\vis$.  
  \end{prop}
  \begin{proof}
    Suppose to the contrary that the set
    $$
K \df \left\{ \mathcal{P} : \mathcal{P} \text{ is discrete and
    countable, and } \mathcal{P}_\vis \neq \mathcal{P}_\star \right\}
$$
satisfies $\nu(K)>0.$ The group $\R^n$ acts on $\mathrm{Cl}(\R^n)$  by
translations, and preserves $\nu$. By 
the 
Birkhoff ergodic theorem applied to the indicator function
$\mathbf{1}_K$, there is $\mathcal{P}_0 \in K$ such that
\begin{equation}\label{eq: positive measure}
  \vol(
\mathbf{T})>0, \ \ \text{ where } \mathbf{T} \df \left\{t \in
  \R^n: t + \mathcal{P}_0 \in K \right\}.
\end{equation}
Write $(\mathcal{P}_0)_\star= \{\mathbf{x}_1, \mathbf{x}_2, \ldots \}$. For
any $i \neq j$, the set
$$
\mathbf{T}_{ij} \df \{t \in \R^n : 0 \text{ is on the line through
} t+\mathbf{x}_i \text{ and } t+ \mathbf{x}_j\}
$$
is a line in $\R^n$, so satisfies $\vol(\mathbf{T}_{ij})=0$. But
whenever $ \mathcal{P}_\vis \neq \mathcal{P}_\star$, 
there are two distinct nonzero points in
$\mathcal{P}$ such that the line containing them passes through the origin. Thus 
we have $\mathbf{T} \subset \bigcup_{i \neq j} \mathbf{T}_{ij}$,
contradicting \eqref{eq: positive measure}. 
    \end{proof}

Let $\SL_n(\R)$ and $\ASL_n(\R)$ denote the groups of
orientation- and volume-preserving linear and affine transformation on
$\R^n$ respectively. Consider the associated homogeneous spaces  
$$\mathscr{X}_n\df \SL_n(\R)/\SL_n(\Z), \qquad
\mathscr{Y}_n \df \ASL_n(\R)/\ASL_n(\Z),$$ 
of covolume-one lattices and grids in $\R^n$. Both
spaces are equipped with the quotient topology, or equivalently, the Chabauty-Fell topology.  
Let $m_{\mathscr{X}_n}$ and $m_{\mathscr{Y}_n}$ denote the Haar-Siegel measures on
$\mathscr{X}_n$ and $\mathscr{Y}_n$, respectively; i.e., the unique
Borel 
probability measures invariant under the transitive action of $\SL_n(\R)$ and
$\ASL_n(\R)$.
Then we have embeddings
$\mathscr{X}_n \subset \mathscr{Y}_n \subset \mathrm{Cl}(\R^n)$.

Fixing the direct sum decomposition
\eqref{eq: direct
  sum decomposition}, the corresponding projections
$\pi_{\mathrm{phys}}, \, \pi_{\mathrm{int}}$, and a window
$\mathcal{W}$, following \cite{Marklof2} we define a map
$$
\Psi: \mathscr{Y}_n \to \mathrm{Cl}(\R^d), \ \  \ \Psi(\mathcal{L})
\df \Lambda(\mathcal{W}, \mathcal{L}).
$$
Finally let $\bar \mu$ and $\mu$ denote respectively the pushforwards
of $m_{\mathscr{Y}_n}$ and $m_{\mathscr{X}_n}$ under $\Psi$. That is,
these measures construct a random cut-and-project set by fixing the
direct sum decomposition and window, and randomly choosing a grid or
lattice $\mathcal{L}$. It was
noted in \cite{Marklof2} that $\bar \mu$ (respectively, $\mu$) is
invariant and ergodic under the action of $\ASL_d(\R)$ (respectively, $\SL_d(\R)$)
on $\mathrm{Cl}(\R^d),$ and gives full mass to the collection of irreducible
cut-and-project sets.  Measures satisfying these properties are
called {\em RMS measures}; they have been completely
classified, see \cite{Marklof2, Barak1}.
When referring to `generic' cut-and-project sets, we have in mind
cut-and-project sets which form a set of full measure with respect to
$\mu$.

A crucial difference 
between $\mu$ and $\bar \mu$ is that $\bar \mu$ is invariant under
translations, but $\mu$ is not. Applying Proposition \ref{prop: MS
  vindicated}, we immediately obtain:

\begin{prop}\label{prop: MS vindicated 1}
  For $\bar \mu$-a.e.\,cut-and-project set, $\Lambda_\star  =
  \Lambda_\vis,$ and thus $\theta (\Lambda) = \theta (\Lambda_\vis).$ 
  \end{prop}

  The situation for $\mu$ is very different, as the following shows:
  
  \begin{prop}\label{prop: ae mu}
    Suppose the window $\mathcal{W}$ is star-shaped with respect to
    the origin. Let $m$ denote the Haar measure on $\mathrm{GL}_n(\R)$. Then
    for $m$-a.e.\,$g\in \mathrm{GL}_n(\R)$ we have 
    \begin{equation}\label{eq: in particular}\Lambda(\mathcal{W},
      g\Z^n_\vis) = \Lambda(\mathcal{W}, g\Z^n)_\vis.
      \end{equation}
      In particular, 
      for
   $\mu$-a.e.\,cut-and-project set $\Lambda$, if $\Lambda =
   \Lambda(\mathcal{W}, \mathcal{L})$ then $\Lambda_\vis =
   \Lambda(\mathcal{W}, \mathcal{L}_\vis).$ 
    \end{prop}

    \begin{proof}
Denote
$$\mathscr{B} \df  \left\{g\in G\, :\, \exists u_1, u_2 \in \Z^n\text{
    s.t.}\, 
  \dim(\textup{span}_{\R}(g u_1, g u_2)\medcap \R^m) \geq 1
\right\}.$$
In light of Proposition \ref{prop: condition for equality}, it suffices to prove that
$m(\mathscr{B})=0$. 
We have $$
\mathscr{B} =\bigcup_{u_1, u_2}\mathscr{B}(u_1,u_2),
$$
where the union ranges over 
pairs of vectors in
$\Z^n$, and 
$$\mathscr{B}(u_1,u_2)=\left\{g\in G\, :\,\dim(\textup{span}_{\R}(g
  u_1, g u_2)\medcap \R^m) \geq 1  \right\}.$$
Thus it is enough to prove $m(\mathscr{B}(u_1,u_2))=0$, for fixed
$u_1,u_2$. The set $\mathscr{B}(u_1, u_2)$ is a
submanifold (in fact, an 
algebraic subvariety) of $\mathrm{GL}_n(\R)$. By our assumption $d
\geq 2$, there is $g \in \mathrm{GL}_n(\R)$ for which $gu_1, gu_2$
both belong to $\R^d$ and in particular 
$\mathscr{B}(u_1,u_2)$ is a proper submanifold of
$\mathrm{GL}_{n}(\R)$.  Recalling that $m$ assigns zero measure to proper submanifolds
of $\mathrm{GL}_n(\R)$, we obtain that $m(\mathscr{B}(u_1,u_2))=0$.

Now let $m'$ denote the Haar measure on $\SL_n(\R)$. Since the
validity of \eqref{eq: in particular} is not affected if one replaces
$\Z^n$ by its dilate $c\Z^n$, \eqref{eq: in
  particular} also holds for $m'$-a.e.\,$g \in \SL_n(\R)$.  Since the measure
      $m_{\mathscr{X}_n}$ is the restriction of $m'$ to a fundamental
      domain for the action of $\SL_n(\R)$, the second assertion
      follows. 
      \end{proof}

Following \cite{Schmidt counting}, we say that a collection of Borel subsets $\{\Omega_T\, :\, T > 0 
\}$ of $\R^d$ is an {\em unbounded ordered family} if 
\begin{itemize}
\item $0\leq T_1\leq T_2 \Rightarrow \Omega_{T_1}\subset \Omega_{T_2}$;
\item For all $T>0$, $\vol(\Omega_T)< \infty$;
\item  $\vol(\Omega_T)\xrightarrow[T\to \infty]{} \infty$; and 
\item For all large enough $V>0$ there is $T$ such that $\vol(\Omega_T)=V$.
\end{itemize}

\begin{thm}\label{theorem on random cut-and-project sets}
  Let $d\geq 2$, and let  the window $\mathcal{W} \subset \R^m$ 
be star-shaped with respect to the origin. 
Fix an unbounded ordered
family $\{\Omega_T : T>0\}$ in $\R^d$. Then, for any
$\varepsilon >0$, for 
$\mu$-a.e. $\Lambda$, 
\begin{equation}\label{equation for error bound for visible set in random cps}
\#(\Omega_T\medcap \Lambda(\mathcal{W},\mathcal{L})_{\vis})=
\frac{\vol(\mathcal{W})}{\zeta(n)}\cdot \vol(\Omega_T) +
O\left(\vol(\Omega_T) ^{\frac{2}{3} +\varepsilon}\right).
\end{equation}
\end{thm}

\begin{proof}
  The proof will proceed by passing to the larger space $\R^n$ and
  applying a result of Fairchild and Han \cite{Han1}.  Recall we have
  assumed $d \geq 2$ and thus $n = d+m  \geq 3$. 
  Let $G \df \SL_n(\R)$, and as before, let $m'$ denote the Haar measure on $G$. Let
  $\{\Omega'_T: T>0\}$ denote an unbounded ordered family in $\R^n$,
  and let $\varepsilon >0$. 
  It follows from the case $N=1$ of \cite[Theorem 1.3]{Han1},
  that for $m'$-a.e.\,$g\in G$,
$$
\#(g\Z^n_{\vis}\medcap \Omega'_T)=
\frac{\vol(\Omega'_T)}{\zeta(n)} +
O\left(\vol(\Omega'_T)^{\frac{2}{3}+\varepsilon}\right).
$$
In particular, if we specialize to $\Omega'_T \df \Omega_T \times
\mathcal{W}$, then we obtain  for $\mu$-a.e.\,$\mathcal{L} \in \mathscr{X}_n$,
\begin{equation*}
  \#(\mathcal{L}_{\vis}\medcap \Omega'_T)
  = \frac{\vol(\Omega_T)\cdot \vol(\mathcal{W})}{\zeta(n)} +
O\left(\vol(\Omega_T)^{\frac{2}{3} + \varepsilon}\right)
\end{equation*}
(where the implicit constant depends on $\mathcal{L}$, $\mathcal{W}$,
$\varepsilon$ and the family $\{\Omega_T\}$).
Note that when $\pi_{\mathrm{phys}}|_{\mathcal{L}}$ is
injective we have
$$
\#(\mathcal{L}_{\vis} \cap \Omega'_T )= \# (\Lambda(\mathcal{W},
\mathcal{L}_\vis )\cap \Omega_T). 
$$

Thus, restricting further to
the set of $\mathcal{L}$ for which the conclusion of Proposition
\ref{prop: ae mu} holds, and for which $\pi_{\mathrm{phys}}|_{\mathcal{L}}$ is
injective, we obtain that for $\mu$-a.e.\,$\mathcal{L}$,
\eqref{equation for error bound for visible set in random 
  cps} holds.
\end{proof}

\begin{remark}
It would be interesting to extend Theorem \ref{theorem on random
  cut-and-project sets} to other RMS measures 
which are not invariant under translations (the case of translation
invariant RMS measures follows from Proposition \ref{prop: MS
  vindicated} and \cite{Barak1}).
  \end{remark}

\section{Quadratic number fields and Hammarhjelm examples}
\label{sec: Hammarhjelm examples}
In this
section, we will recall relevant information about real quadratic
number fields and review the results of Hammarhjelm \cite{Gustav}. Let $K=
\Q(\sqrt{d})$ with $d>1$, be a real quadratic number field, and let
$\mathcal{O}_K$ be its ring of integers. The {\em norm} of an integer $x\in
\mathcal{O}_K$ is defined as  
$$N(x)=x\sigma(x),$$
where $\sigma$ is the non-trivial Galois automorphism of $K$. The {\em
  unit
group} $\mathcal{O}_K^{\times}$ consists of the elements 
$x \in \mathcal{O}_K$ such that $N(x)=\pm 1$. The unique element $\lambda = \lambda_K
\in \mathcal{O}_K^{\times}$ such that 
$$\mathcal{O}_K^{\times}= \{\pm 1\}\times \{\lambda^i\, :\, i\in
\Z\},\qquad \lambda>1,$$ 
is called the {\em fundamental unit} of $K$.
For an ideal $I\in \mathcal{O}_K$, the {\em norm} of $I$ is defined as
$$N(I)= \left|\mathcal{O}_K/I\right|,$$
which is always finite. 
If $\mathcal{O}_K$ is a principal ideal domain, we have $N(I)=
|N(x)|$, where $I$ is the principal ideal generated by $x$. Throughout
the paper, we will consider fields $K$  for which 
 $\mathcal{O}_K$ is a principal ideal domain. For any $x_1, \ldots, 
 x_d \in \mathcal{O}_K$, let $\gcd(x_1,\ldots, x_d)$ denote a fixed generator
 of the ideal generated by $x_1, \ldots, x_d$. We write $\gcd(x_1,
 \ldots, x_d)=1$
 when $\gcd(x_1, \ldots, x_d)$  is a unit. We say that $x \in \mathcal{O}_K$ is {\em
   prime} if $x = ab$, $a, b \in \mathcal{O}_K \implies a \in
 \mathcal{O}_K^\star \text{ or }  b \in
 \mathcal{O}_K^\star.$

The Dedekind zeta function associated with $\mathcal{O}_K$ is given by
$$\zeta_{\mathcal{O}_K}(s) = \sum_{I\subset \mathcal{O}_K}\frac{1}{N(I)^s},$$
where the summation is over the nonzero ideals $I\subset
\mathcal{O}_K$ and the series converges whenever $\Re(s)>1$
(see \cite[p.131]{Marcus} for details).  
The M\"obius function $\mu$ is defined for ideals $I\subset \mathcal{O}_K$ as follows:
\begin{equation*}
\begin{split}
\mu (I)&= \begin{cases}
    1,&\qquad N(I)=1,\\
    (-1)^r,& \qquad I= p_1\cdots p_r,\, \textup{for distinct prime ideals}\, p_1,\ldots,p_r,\\
    0, & \qquad I\subset p^2 \, \textup{for some prime ideal p}.
\end{cases}    
\end{split}    
\end{equation*}
Using the M\"obius function, the reciprocal of the Dedekind zeta
function can be expressed as 
\begin{equation}\label{equation of zeta(K, d)}
\frac{1}{\zeta_{\mathcal{O}_K}(s)} = \sum_{I\subset
  \mathcal{O}_K}\frac{\mu(I)}{N(I)^s},  
\end{equation}
and this series converges absolutely when $\Re(s)>1$. 
Since any ideal in $\mathcal{O}_K$ can be written as $I = g\mathcal{O}_K$, where $g \in
\mathcal{O}_K$, and we will sometimes write $\mu(I)$ as $\mu(g)$.  
The {\em Minkowski embedding} of $\mathcal{O}_K$ into $\R^2$ is 
defined as 
$$\mathcal{L}_{\mathcal{O}_K} \df \{(x,\sigma(x))\, : \, x\in \mathcal{O}_K\}.$$
We say that $\mathcal{W} \subset \R^m$ is {\em
  centrally symmetric} if $\mathcal{W} = - \mathcal{W}.$
With these preliminaries, we can state the result of Hammarhjelm 
\cite{Gustav}.

\begin{thm}\label{thm: Gustav density}
Let $\mathcal{W} \subset \R^2$ be star-shaped with respect to the  origin and 
centrally symmetric, let $K$ be one of $\Q(\sqrt{2}), \Q(\sqrt{5})$,
and let
$$\mathcal{L} \df 
\{(x_1, x_2, \sigma(x_1), \sigma(x_2)): x_1, x_2 \in
\mathcal{O}_K\} \cong \mathcal{L}_{\mathcal{O}_K} \oplus \mathcal{L}_{\mathcal{O}_K}.$$
Then the density of $\Lambda(\mathcal{W}, \mathcal{L})_\vis$ exists with
respect to any Jordan measurable averaging set $D$, and we have 
$$\theta(\Lambda(\mathcal{W}, \mathcal{L})_\vis) = \left(1- \frac{1}{\lambda_K^2}
\right)\, \frac{1}{\zeta_{\mathcal{O}_K(2)}}\, \theta(\Lambda(\mathcal{W}, \mathcal{L})).$$
  \end{thm}
  \begin{remark}\label{remarks: three}
1. For specific choices of $\mathcal{W}$, one obtains the
Amman-Beenker point set and sets associated with the Penrose tiling vertex set. The density
of visible points in the Amman-Beenker point set had been established earlier by Sing
\cite{Sing_slides}.

2. Hammarhjelm's results are stated with a slightly less restrictive condition
than central symmetry, namely, for an explicit constant $c>1$ depending
on the field $K$, Hammarhjelm requires $-\mathcal{W} \subset c
\mathcal{W}$.

3. The proof of Theorem \ref{thm: Gustav density} works for other
fields which satisfy a condition we will discuss below.
    \end{remark}
\subsection{The Hammarhjelm condition}\label{subsec: Hammarhjelm examples}
Let $K=\Q(\sqrt{d})$ be a real quadratic number field with $d\geq 2$
square free. In the range $2\leq d\leq 100$, for  
\begin{equation*}
\begin{split}
  d&= 2,3,5,6, 7,11,13,14,17,19,21,22,23,29,31,33,37,38,41,43, \\
  &~~~~46,47,
53,57,59,61,62,67,69,71,73,77,83,86,89,93,94,97
\end{split}    
\end{equation*}
the ring of integers $\mathcal{O}_{K}$ are principal ideal domains. It
is conjectured that there are infinitely many real quadratic number
fields $K$ with $\mathcal{O}_K$ principal ideal domain (see
\cite[p.37]{Neukirch1}). Denote 
$$\mathcal{P}= \{\pi \in \mathcal{O}_K\, :\, \pi \, \textup{is prime
  and }\,1<\pi <\lambda \},$$
where $\lambda$ is the fundamental unit. We say that a
number field $K$ for which $\mathcal{O}_K$ is principal ideal domain
{\em satisfies the Hammarhjelm condition} if 
$$|\sigma(\pi)|>1,\qquad \pi \in \mathcal{P}.$$
Hammarhjelm
\cite{Gustav} verified that the condition holds for $\Q(\sqrt{2})$ and
$\Q(\sqrt{5}).$

The condition may be 
interpreted graphically as follows. Let $\mathcal{L}_{\mathcal{O}_K}$
be the Minkowski embedding of $\mathcal{O}_K$. 
Then the Hammarhjelm condition holds precisely when 
\begin{equation}\label{eq: Hammarhjelm condition}
 \mathcal{L}_{\mathcal{O}_K}\medcap
 \left((1,\lambda) \times [-1,1]\right) = \varnothing.
\end{equation}
Note that the left-hand side of \eqref{eq: Hammarhjelm condition} is
always finite, being the intersection of a lattice and a bounded
set. Furthermore, generators for $\mathcal{L}_{\mathcal{O}_K}$ can be
easily computed. Thus one can check \eqref{eq: Hammarhjelm condition}
by hand (see Figure \ref{fig 
  0.1}). In the range $1\leq d\leq 100$, the Hammarhjelm condition is
satisfied only for $d=2,5,13,29,53$.

\begin{figure}
\centering
\includegraphics[width=0.50\linewidth]{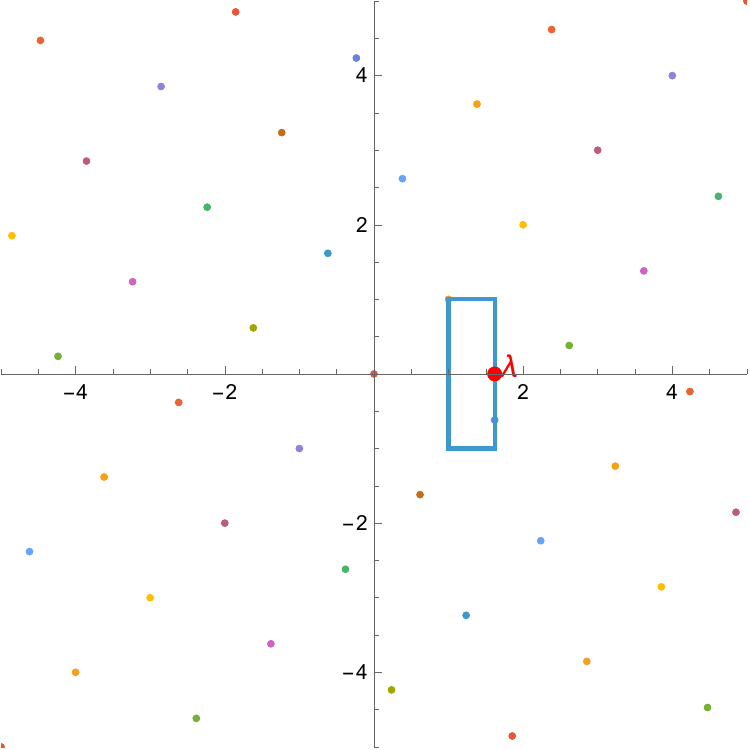}
 \caption{The Minkowski embedding of $\mathcal{O}_K$, where
   $K=\Q(\sqrt{5})$ and $\lambda= \frac{1}{2}(1+\sqrt{5})$.} 
 \label{fig 0.1}
\end{figure}

We will say that a cut-and-project set $\Lambda(\mathcal{W},
\mathcal{L})$ is a {\em Hammarhjelm example} if $\mathcal{W}$ is
convex and centrally symmetric, and
$\mathcal{L} = \mathcal{L}_{\mathcal{O}_K} \oplus
\mathcal{L}_{\mathcal{O}_K}$, where $K$ is a real quadratic field $K$ satisfying the Hammarhjelm condition.

\section{Effective inclusion-exclusion}
In this section we prove Theorem \ref{thm: Gustav examples}. Once more, we prove a
stronger result, namely:

\begin{thm}\label{theorem on error term}
Let $d \geq 2$, let $m=d$, and let $\mathcal{W}\subset \R^{d}$ be a
convex centrally 
symmetric window.  
Let $K$ be a real quadratic field satisfying the Hammarhjelm condition, and let 
\begin{equation}\label{equation of the lattice}
\mathcal{L} \df  \{(x_1,\ldots, x_d,\sigma(x_1), \ldots, \sigma(x_d))\, :
\, x_i \in \mathcal{O}_K\} \cong \bigoplus^d \mathcal{L}_{\mathcal{O}_K} \subset \R^{2d}.
\end{equation}
Let $\Lambda = \Lambda(\mathcal{W}, \mathcal{L})$, and let $D \subset
\R^d$ be a  convex averaging set. Then  we have
the following asymptotic estimate as $T \to \infty$
\begin{equation*}
\begin{split}
\frac{\#(\Lambda_{\vis} \medcap
  TD)}{\textup{vol} (TD)}= \left(1-\frac{1}{\lambda^d}\right)\cdot
\frac{\theta
  (\Lambda)}{\zeta_{\mathcal{O}_K}(d)}
+ \begin{cases} 
    O\left(\frac{\log T}{\sqrt{T}}\right),& \qquad d=2,\\
    O\left(\frac{1}{\sqrt{T}}\right),& \qquad d\geq 3.
\end{cases}
\end{split}
\end{equation*}
  \end{thm}

It is possible to relax the convexity assumption on  $\mathcal{W}$ and
$D$ in Theorem \ref{theorem on error term} as
follows. If there is some $s \in \N$ so that $D$, $\mathcal{W}$ and $D \times t \mathcal{W}$ (for every
$t$), are all of {\em narrow class 
  $s$} then in the proof below, we may apply \cite[Lemma 1]{Schmidt
  counting convex} (see also \cite[Thm. 2.3]{Widmer}) as a substitute for
Proposition \ref{prop: Schmidt counting convex sets} below. We leave the
details to the dedicated reader.

  In the proof we will use the following counting result, see
  \cite{Schmidt counting convex}:

  \begin{prop}\label{prop: Schmidt counting convex sets}  For any $n
    \in \N$
    there is $C>0$ so that the following holds. 
Let $\mathcal{S} \subset \R^n$, let $\mathcal{L} \subset \R^n$ be a
lattice, and let $c>0$ and $ T_0 \geq 1$ such that:
\begin{enumerate}
\item \label{item: Schmidt 1}
  $\mathcal{S}$ is convex and its diameter is bounded above by $T_0$;
\item \label{item: Schmidt 2}
  The lattice $\mathcal{L}$ contains $n$ linearly independent
    vectors of length at most $T_0$, and $n-1$ linearly independent
    vectors of length at most $c$.
  \end{enumerate}
  Then
  \begin{equation}\label{eq: asymptotic Schmidt count}
  \left| \#(\mathcal{S} \cap \mathcal{L})  -
  \frac{\vol(\mathcal{S})}{\mathrm{covol}(\mathcal{L})} \right| \leq C
  cT_0^{n-1}.\end{equation} 
    \end{prop}

    \begin{remark}
     Note that in \cite{Schmidt counting convex} it is further assumed
    that $\mathcal{S}$ is compact; however for convex sets, the
    general case can be obtained by approximating $\mathcal{S}$ from
    inside and outside by compact convex sets.
  \end{remark}


  For $ g\in \mathcal{O}_K,$ let $a_g$ be the diagonal matrix
  $$a_g \df  \textup{diag}(\underbrace{g,\ldots, g}_\text{$d$ times},
\underbrace{\sigma(g),\ldots, \sigma(g)}_\text{$d$ times}),$$ 
and let 
\begin{equation}\label{equation of lattice L(g)}
  \mathcal{L}_g \df
  a_g \mathcal{L}
  =
  \{(g x_1,\ldots, g x_d,  \sigma(g
x_1),\ldots,\sigma(gx_d))\, : \, x_i \in \mathcal{O}_K\}. 
\end{equation}
Since $\mathcal{O}_K$ is a ring, $\mathcal{L}_g \subset \mathcal{L}$,
and we have
\begin{equation}\label{eq: covolume g}
  \mathrm{covol}(\mathcal{L}_g) =  |N(g)|^d \, \mathrm{covol}(\mathcal{L}).
  \end{equation}

{\bf Notation.} In the remainder of this section, $K$ and
$\mathcal{W}$ satisfy the hypotheses of Theorem \ref{theorem on error
  term}, $\mathcal{L}$ is given by \eqref{equation of the lattice}, and $\beta  \in \mathcal{O}_K^\times$.

\begin{lemma}\label{lemma on points of cps (L(g), beta W) in a ball}
  Let
  $\mathcal{L}_g$  be as in
  \eqref{equation of lattice L(g)},
and let $D \subset \R^d$ be a convex averaging
set. Then, there is a constant $C_1$ such that for any $g \in
\mathcal{O}_K$, and any $T \geq |N(g)|$, 
$$\left|\#\left(\Lambda (\beta \mathcal{W}, \mathcal{L}_g)\medcap
    TD\right) - \frac{\vol(TD)\cdot
    \vol(\mathcal{W})}{\mathrm{covol}(\mathcal{L})}\cdot
  \frac{\beta^{d}}{|N(g)|^d}\right| \leq C_1
\left(\frac{T}{|N(g)|}\right)^{d-\frac12}.$$ 
\end{lemma}

\begin{proof}
We have
\begin{equation}\label{mywedge intersects with B(R) 1}
\begin{split}
  \#\left(\Lambda (\beta \mathcal{W}, \mathcal{L}_g)\medcap TD
  \right)
   &
  = \#\left((TD \times \beta \mathcal{W})\medcap
   \mathcal{L}_g\right). \\ 
&= \#\left(a_g^{-1}\left(TD \times \beta \mathcal{W}\right)\medcap \mathcal{L}\right).
\end{split}   
\end{equation}
This reduces our problem to the problem of counting lattice points in
a convex set. Our goal will be to ensure that the hypotheses of
Proposition \ref{prop: Schmidt counting convex sets} are satisfied,
with $T_0$ as small as possible. To
this end we will apply diagonal elements which fix the lattice
$\mathcal{L}$ but change the convex body in question.

Choose positive numbers $c, \, T_1$ so that assumption \eqref{item: Schmidt 2} of
Proposition \ref{prop: Schmidt counting convex sets} holds, for $c$
and for any $T_0 \geq T_1$. 
Let $A_1$ denote the one-parameter diagonal group 
$$A_1\df  \left\{\textup{diag}(\underbrace{e^s,\ldots,
    e^s}_\text{$d$ times}, \underbrace{e^{-s},\ldots,
    e^{-s}}_\text{$d$ times})\, :\, s \in \R\right\}.$$
Now suppose $g_0 \in \mathcal{O}_K^\times , \, g_0>1$ is
a unit of norm one. Then 
multiplication by $g_0$ permutes the integers of $K$ and hence
$a_{g_0}\mathcal{L}=\mathcal{L}$.
Moreover 
$a_{g_0} \in A_1$.
It follows that
$$A_0 \df \left\{a_0 \in A_1 \, : \, a_0 \mathcal{L} = \mathcal{L}  \right\}$$ 
is a co-compact subgroup of $A_1$,
and for any $a_0 \in A_0$, 
$$
\#\left(\Lambda (\beta \mathcal{W}, \mathcal{L}_g)\medcap TD
  \right)
  \stackrel{\eqref{mywedge intersects with
  B(R) 1}}{=} \#a_0 \left(a_g^{-1}(TD \times \beta \mathcal{W})\medcap
   \mathcal{L}\right) =  \# \left(a_0a_g^{-1} (TD \times \beta \mathcal{W})\medcap
   \mathcal{L}\right).
$$

In the remainder of the proof, we will say that two quantities $X,Y$ are {\em comparable} if
their ratio $X/Y$ is bounded above and below by positive constants which do
not depend on $g$ and $T$ (but may depend on the window $\mathcal{W}$,
the averaging set $D$, the number  
field $K$,  and the number $\beta$). In this case we will write $X
\asymp Y$.

For $a_0 \in A_0$ we will write
$$\mathcal{S}(a_0) \df
a_0a_g^{-1} (TD \times \beta \mathcal{W})  = \mathcal{S}_1(a_0)
\times \mathcal{S}_2(a_0), \ \ \ \ \text{where } \mathcal{S}_i(a_0)
\subset \R^d$$
(note that these sets also depend on $T$ but we suppress this
dependence in the notation). 
Since $A_0$ is cocompact in $A_1$, and the linear action of $A_1$
scales the first factor of the decomposition $\R^{2d}  = \R^d \oplus \R^d$ by a constant factor while scaling the
second factor by its
reciprocal, we can find $a_0 \in A_0$ so that $\mathcal{S}_1(a_0)$ and
$\mathcal{S}_2(a_0)$ 
have comparable diameters $D_1, D_2$. Moreover each diameter $D_i$ is
comparable to the $d$th root of the volume $\vol(\mathcal{S}_i(a_0))$. Since
$$
D_1^d \, D_2 ^d \asymp \vol (\mathcal{S}(a_0)) = \vol(a_0 a_g^{-1}(TD \times \beta
\mathcal{W}))  = \vol(a_g^{-1}(TD \times \beta
\mathcal{W})) \asymp  \frac{ T^d 
}{|N(g)|^{d}},$$
this implies that
\begin{equation}\label{eq: same order}
  D_1 \asymp D_2 \asymp \sqrt{\frac{T}{|N(g)|}}
  \asymp \mathrm{diam}(\mathcal{S}(a_0)).
  \end{equation}
 Thus both assumptions of Proposition \ref{prop:
  Schmidt counting convex sets} hold, with $n=2d$ and with
\begin{equation}\label{eq: def T0} T_0 \df  
  \max \left( T_1, \mathrm{diam}(\mathcal{S}(a_0)) \right)
  .
\end{equation}
Note that $T_0 \asymp \sqrt{\frac{T}{|N(g)|}}
$. Indeed, 
when the maximum in \eqref{eq:
  def T0} is attained by the term $\mathrm{diam}(\mathcal{S}(a_0))$,
we have this from \eqref{eq: same order}, and when the maximum is
attained by $T_1$, we have this from $1 \leq
\sqrt{\frac{T}{|N(g)|}} \leq T_1 = T_0$. 
Now the desired conclusion follows from 
the conclusion of Proposition \ref{prop:
  Schmidt counting convex sets}. 
\end{proof}

The following lemma was proved in the case $d=2$ in \cite[Lemma
4.6]{Gustav}).

\begin{lemma}\label{lemma on upper bound for the number of points}
  Let
  $g \in \mathcal{O}_K$, let
  $T >0$ and let  $B_T$ be the ball of radius $T$
around the origin. Then for $\beta \in \mathcal{O}_K^\times$, there is a
constant~$L$ such that  
$$\#(\Lambda(\beta \mathcal{W}), \mathcal{L}_g)_\star\medcap B_T) 
\leq \frac{L \cdot T^d}{|N(g)|^d}, 
$$ 
where the constant $L$ depends only on $\mathcal{W}$ and $\beta$.
\end{lemma}
The  proof of Lemma \ref{lemma on upper bound for the number of
  points} again uses the invariance of $\mathcal{L}$ under the group
$A_0$ and a rescaling argument, as in the proof
of Lemma \ref{lemma on points of cps (L(g), beta W) in a ball}. The
proof in \cite{Gustav} easily generalizes to arbitrary $d \geq 2$ and we omit it.

\begin{lemma}\label{lemma on set of visible points}
Write $\sigma(x) \df  (\sigma(x_1),\ldots, \sigma (x_d))$. Then we have
  \begin{equation}\label{eq: inclusion subset}
\Lambda(\mathcal{W},\mathcal{L})_{\vis}
=  \left\{(x_1,\ldots, x_d) \in\Lambda(\mathcal{W},\mathcal{L})_\star\, :
  \, \gcd(x_1,\ldots, x_d)=1, \, \sigma(x)
  \not\in \frac{1}{\lambda} \mathcal{W}\right\}.
\end{equation}
\end{lemma}
Similar results were proved in \cite[Propositions 4.8 \&
4.14]{Gustav} for \(d=2\), $K=\Q(\sqrt{2})$,  and $K=\Q(\sqrt{5})$. We
provide the proof for completeness.

\begin{proof}
We start by proving the inclusion $\subset$ in \eqref{eq: inclusion subset}. Suppose
that $x= (x_1,\ldots, x_d)\in \Lambda(\mathcal{W},\mathcal{L})_\star$, and
that $\gcd(x_1,\ldots, x_d)\neq 1$. Then there is a prime $\pi \in
\mathbb{P}$ which divides $x_1,\ldots, x_d$. Consequently  we have
$\pi^{-1} x \in \mathcal{O}_K^d$. 
Replacing $\pi$ if necessary by its multiple by a unit, we may assume
that $\pi \in (1, \lambda)$, and hence,  by the Hammarhjelm condition,  $|\sigma(\pi)|>1$. Since 
$\mathcal{W}$ is convex and centrally symmetric, it follows that  
$$\sigma(\pi^{-1} x)= \sigma(\pi)^{-1} \sigma(x) \in \mathcal{W}.$$ 
Thus, we conclude that $\pi^{-1} x  \in
\Lambda(\mathcal{W},\mathcal{L})$, and $x \notin
\Lambda(\mathcal{W},\mathcal{L})_{\vis}.$ If $\sigma(x) \in
\frac{1}{\lambda} \, 
\mathcal{W}$ then
$$\sigma(\lambda^{-1} x) = \pm  \lambda \sigma (x) \in 
\mathcal{W},$$
where we have used that $\lambda \sigma(\lambda) = \pm 1$ and
$\mathcal{W}$ is convex and centrally symmetric. It follows that $\lambda^{-1} x \in
\Lambda(\mathcal{W},\mathcal{L})$, and again 
$x \notin \Lambda(\mathcal{W},\mathcal{L})_{\vis}.$

Next, we prove the inclusion $\supset$. Suppose that 
$x= (x_1,\ldots,x_d) \neq 0$ does not belong to the left-hand side of
\eqref{eq: inclusion subset}, that is $$x\in
\Lambda(\mathcal{W},\mathcal{L})_\star\smallsetminus
\Lambda(\mathcal{W},\mathcal{L})_{\vis}.$$ 
Then there is some $t \in (0,1)$ such that $tx \in
\Lambda(\mathcal{W},\mathcal{L})$. Let $i$ be an index such that $x_i
\neq 0$. Since the coordinates
$x_i$ and $tx_i$ are both in $\mathcal{O}_K$, we have that 
$t \in
K$. Since $\Lambda(\mathcal{W},\mathcal{L})$ is locally finite, by
taking $t$ as small as possible we may
assume that
$$y = (y_1, \ldots, y_d) \df tx \in
\Lambda(\mathcal{W},\mathcal{L})_{\vis}.$$
Since $t\in K$, we
can write $t =b/a$, with $a,b \in \mathcal{O}_K$, with $a, b$
co-prime. If $b$ is not a unit, then from $ay=bx$ we see that $\gcd(y_1,\ldots, y_d)\neq 1$,
which is a contradiction to the direction $\subset$ we have already
proved. 
Thus we can assume $a = t^{-1}\in \mathcal{O}_K$, and so for
each $i$, $a$ divides $x_i = t^{-1}y_i$. Thus, either 
$\gcd(x_1,\ldots, x_d)\neq 1$, in which case $x$ does not belong to the
right-hand side of \eqref{eq: inclusion subset}, 
or $a$ is a unit. If $a$ is a unit then 
$t=\lambda^{-k}$ for some $k \in \Z$, and since $t \in (0,1)$ and
$\lambda>1$  we have
$k \in \N$.
We have that $\mathcal{W}$ contains $\sigma(x)$ as well as $\sigma(y) = \sigma(\lambda^{-k}x)
= \pm \lambda^k \sigma(x)$. Since $\mathcal{W}$ is centrally
symmetric and convex, it follows that $\mathcal{W}$ also contains
$\lambda \sigma(x)$, that is, $\sigma(x) \in \frac{1}{\lambda} \,
\mathcal{W}$. So in this case again we have that  $x$ does not belong to the
right-hand side of \eqref{eq: inclusion subset}. 
\end{proof}

For the lattice $\mathcal{L}$  as in \eqref{equation of the lattice}
and a window $\mathcal{W}$, define the set of {\em primitive points} as 
$$\Lambda_{\textup{pr}}(\mathcal{W}, \mathcal{L}) \df \{x = (x_1,\ldots,
x_d) \in \Lambda(\mathcal{W}, \mathcal{L})_\star\, : \, \gcd(x_1,\ldots,
x_d)=1
\}.$$ 
Note that for the integer lattice, with the standard gcd, the set of primitive points
coincides with the set of visible points. But here we work with the
gcd of the quadratic field $K$ and this is no longer the case. Lemma
\ref{lemma on set of visible points} shows that under the Hammarshjelm condition,
 \begin{equation}\label{eq: under H
     condition}\#\left(\Lambda(\mathcal{W},\mathcal{L})_{\vis}\medcap
     TD \right)= 
\#\left(\Lambda_{\textup{pr}}(\mathcal{W},\mathcal{L})\medcap TD \right)-
\#\left(\Lambda_{\textup{pr}}\left(\frac{1}{\lambda } \, 
    \mathcal{W},\mathcal{L} \right)\medcap TD \right).
\end{equation}

Given $T>0$, 
$\beta \in \mathcal{O}_K^\times$, and an
averaging set $D$, let
$$C \df \left\{\pi \in \mathbb{P} : \Lambda(\beta \mathcal{W},\mathcal{L})_{\star}\medcap\Lambda(\beta
\mathcal{W},\mathcal{L}_{\pi})_{\star}\medcap TD \neq
\varnothing \right \}.$$
It follows from the local finiteness of cut-and-project sets (see
\cite[Lemma 4.4]{Gustav} for the detailed argument) that $C$ 
is finite, and we write $C =\{
\pi_1,\ldots, \pi_n\}.$

We have
\begin{equation}\label{pri points inclusion-exclusion1}
\begin{split}
  & 
  \Lambda_{\textup{pr}}(\beta \mathcal{W},\mathcal{L})\medcap
  TD 
  = 
  \left(\Lambda(\beta
    \mathcal{W},\mathcal{L})_{\star}\medcap TD\right) \smallsetminus
  \bigcup_{\pi \in \mathbb{P}}  \left(\Lambda(\beta
    \mathcal{W},\mathcal{L})_{\star}\medcap\Lambda(\beta
    \mathcal{W},\mathcal{L}_{\pi})_{\star} 
  \right) 
  \\ 
  =& 
  \left(\Lambda(\beta \mathcal{W},\mathcal{L})_{\star}\medcap
    TD\right) \smallsetminus \bigcup_{i=1}^n  \left(\Lambda(\beta
    \mathcal{W},\mathcal{L})_{\star}\medcap\Lambda(\beta
    \mathcal{W},\mathcal{L}_{\pi_{i}})_{\star}\medcap
    TD\right), 
\end{split}    
\end{equation}
For a finite subset $F\subset \mathbb{P}$, let $\prod_{F}$ denote the
product of elements of $F$. We have 
\begin{equation}\label{Equation on the intersection of cut and project sets}
\Lambda(\beta\mathcal{W}, \mathcal{L})\medcap \left(\bigcap_{\pi \in
    F}\Lambda(\beta\mathcal{W}, \mathcal{L}_{\pi})\right)=
\Lambda\left(\beta\mathcal{W}, \mathcal{L}_{\prod_F} \right).     
\end{equation}
Since $\mathcal{L}_{g}= \mathcal{L}_{ug}$ for any unit $u$ and $g\in \mathcal{O}_K$, we have
\begin{equation}\label{equation on two equal cps}
\Lambda(\mathcal{W}, \mathcal{L}_{ug})=  \Lambda(\mathcal{W}, \mathcal{L}_{g}).  
\end{equation}
By \eqref{pri points inclusion-exclusion1}, \eqref{Equation on the
  intersection of cut and project sets}, \eqref{equation on two equal
  cps}, and the inclusion and exclusion principle, we have 
\begin{equation}\label{equation for inclusion and exclusion principle1}
\begin{split}
\#(\Lambda_{\textup{pr}}(\beta \mathcal{W}, \mathcal{L})\medcap TD)&=
\# \left(\left(\Lambda(\beta \mathcal{W},\mathcal{L})_{\star}\medcap
    TD\right) \right)\\ 
&~~+ \sum_{i=1}^n (-1)^i \sum_{F_i\subset C: |F_i|=i}
\#\left(\left(\left(\Lambda(\beta\mathcal{W}, \mathcal{L})_{\star}
      \medcap \bigcap_{\pi\in F_i}\Lambda(\beta \mathcal{W},
      \mathcal{L}_{\pi})_{\star}\right)\medcap TD\right)\right) \\ 
&= \# \left(\left(\Lambda(\beta
    \mathcal{W},\mathcal{L})_{\star}\medcap TD\right) \right)\\ 
&~~+ \sum_{i=1}^n (-1)^i \sum_{F_i\subset C: |F_i|=i}
\#\left(\left(\Lambda\left(\beta \mathcal{W},
    \mathcal{L}_{\prod_{F_i}} \right)_{\star}\medcap TD\right)\right) \\ 
&=\sum_{g \in [1,\lambda)\medcap \mathcal{O}_K} \mu(g) \cdot
\#\left((\Lambda(\beta \mathcal{W}, \mathcal{L}_g)_{\star}\medcap
  TD)\right).     
\end{split} 
\end{equation}

\begin{lemma}\label{lemma of primitive points on cps}
  We have
  \begin{equation*}
\begin{split}
\left|\#(\Lambda_{\textup{pr}}(\beta \mathcal{W}, \mathcal{L})\medcap
  TD)- \frac{\vol(\mathcal{W})\cdot
    \vol(TD)}{\mathrm{covol}(\mathcal{L})}\cdot  
\frac{\beta^d}{\zeta_{\mathcal{O}_K}(d)}\right| \leq
&\begin{cases}
  C_2T^{3/2}\log T,&\qquad d=2,\\
   C_d T^{d-1/2},&\qquad d\geq 3,
\end{cases}
\end{split} 
\end{equation*}
where $C_2$ and $C_d$ are constants depending only on $\mathcal{W}$ and $\beta$.
\end{lemma}

\begin{proof}
To simplify notation, denote
\begin{equation}
\label{eq: def MT}
  M_T\df \frac{\vol(\mathcal{W})\cdot
    \vol(TD)}{\mathrm{covol}(\mathcal{L})}.
  \end{equation}
Let $H_n$ be the number of ideals in $\mathcal{O}_K$ of norm $n$. It
follows from \cite[Theorem 39]{Marcus} that 
\begin{equation}\label{equation for number of ideals} 
  H_n\leq H \cdot n^{1/2},  
\end{equation}
where $H>0$ is a constant independent of $n$. 
Since \eqref{equation of zeta(K, d)} converges absolutely and
\eqref{equation for inclusion and exclusion principle1} is a finite
sum, we have
\begin{equation}\label{primitive lattice points equation for bound in cps}
\begin{split}
 \left|\#(\Lambda_{\textup{pr}}(\beta \mathcal{W}, \mathcal{L})\medcap
   TD)- \frac{M_T\beta^d}{\zeta_{\mathcal{O}_K}(d)}\right|& =
 \left|\sum_{g\in [1,\lambda)\medcap
     \mathcal{O}_K}\mu(g)\left(\#(\Lambda(\beta \mathcal{W},
     \mathcal{L}_g)_{\star}\medcap TD)-
     \frac{M_T\beta^d}{|N(g)|^d}\right)\right|.\\ 
\end{split} 
\end{equation}
Note the conclusion of Lemma \ref{lemma on upper bound for the number
  of points} remains valid for $TD$ instead of $B_T$ (perhaps at the
cost of changing the constant $L$). Thus, using Lemmas \ref{lemma on
  points of cps (L(g), beta W) in a ball} 
and 
\ref{lemma on upper bound for the number of points}, the right-hand
side of \eqref{primitive lattice points equation for bound in cps} is
bounded by 
\begin{equation}\label{right-land side of primitive lattice points equation for bound in cps}
\begin{split}
 & C_1 \sum_{\substack{g\in [1,\lambda)\medcap \mathcal{O}_K\\ 
 |N(g)|\leq T}}\frac{T^{d-1/2}}{|N(g)|^{d-1/2}} + L
\sum_{\substack{g\in [1,\lambda)\medcap \mathcal{O}_K\\  
 |N(g)|> T}}\frac{T^{d}}{|N(g)|^{d}}+ M_T\beta^d \sum_{\substack{g\in [1,\lambda)\medcap \mathcal{O}_K\\ 
 |N(g)|> T}}\frac{1}{|N(g)|^d}\\  
  \leq & C_1 \sum_{\substack{g\in [1,\lambda)\medcap \mathcal{O}_K\\ 
 |N(g)|\leq T}}\frac{T^{d-1/2}}{|N(g)|^{d-1/2}} + L \sum_{\substack{g\in [1,\lambda)\medcap \mathcal{O}_K\\ 
 |N(g)|> T}}\frac{T^{d}}{|N(g)|^{d}}+ C_{M,\beta}\cdot T^d
\sum_{\substack{g\in [1,\lambda)\medcap \mathcal{O}_K\\  
 |N(g)|> T}}\frac{1}{|N(g)|^d}\\
  \leq & C_1 \sum_{\substack{g\in [1,\lambda)\medcap \mathcal{O}_K\\ 
 |N(g)|\leq T}}\frac{T^{d-1/2}}{|N(g)|^{d-1/2}}+ C_{M,\beta}'\cdot T^d
\sum_{\substack{g\in [1,\lambda)\medcap \mathcal{O}_K\\  
 |N(g)|> T}}\frac{1}{|N(g)|^d},\\
\end{split} 
\end{equation}
where $C_{M, \beta}$ is a constant depending only on $\mathcal{W}$ and
$\beta$ such that $M_T \beta^d\leq C_{M,
  \beta}\cdot T^d$, 
and $C_{M, \beta}' = C_{M, \beta} +L$. By \eqref{equation for number
  of ideals}, \eqref{primitive lattice points equation for bound in
  cps}, and \eqref{right-land side of primitive lattice points
  equation for bound in cps}, we have 
\begin{equation}\label{equation for bound in cps 2}
\begin{split}
 \left|\#(\Lambda_{\textup{pr}}(\beta \mathcal{W}, \mathcal{L})\medcap
   TD)- \frac{M_T\beta^d}{\zeta_{\mathcal{O}_K}(d)}\right| 
 &\leq  C_1\cdot  \sum_{k=1}^{\lfloor T\rfloor}\frac{H_k \cdot
   T^{d-1/2}}{k^{d-1/2}}+ C_{M, \beta}' T^d \sum_{k=\lfloor
   T\rfloor}^{\infty}\frac{H_k}{k^d}\\ 
 &\leq C_1 H T^{d-1/2}\int_{x=1}^T \frac{x^{1/2}}{x^{d-1/2}}d x+
 C_{M,\beta}'\cdot  H T^{d}\int_{x=T}^{\infty} \frac{x^{1/2}}{x^{d}}d
 x.
\end{split} 
\end{equation}
In case $d=2$, the right-hand side of \eqref{equation for bound in
  cps 2} is
\begin{equation}\label{proof of lemma for primitive lemma error bound
    d equal 2} 
\begin{split}
  & C_1H T^{3/2}\int_{x=1}^T \frac{1}{x}d x + C_{M,\beta }' T^2 H \int_{x=T}^{\infty} \frac{1}{x^{3/2}}d x\\
   \leq &C_1H T^{3/2}\log T + C_{M,\beta }'H T^{3/2}\\
    \leq &C_2(\beta) T^{3/2}\log T,
\end{split}    
\end{equation}
where $C_2(\beta)$ is a constant depending only on $\mathcal{W}$ and
$\beta$. 
In case $d\geq 3$, the right-hand side of \eqref{equation for bound in cps 2} is
\begin{equation}\label{proof of lemma for primitive lemma error bound d greater equal 3}
\begin{split}
  & C_1H T^{d-1/2}\int_{x=1}^T \frac{x^{1/2}}{x^{d-1/2}}d x+
  C_{M,\beta }' T^d H \int_{x=T}^{\infty} \frac{x^{1/2}}{x^{d}}d x\\ 
=& C_1H T^{d-1/2}\int_{x=1}^T \frac{1}{x^{d-1}}d x+ C_{M,\beta }' T^d
H \int_{x=T}^{\infty} \frac{1}{x^{d-1/2}}d x\\ 
\leq & C_1H T^{d-1/2} + C_{M,\beta }' H T^{3/2} \leq C_d(\beta) T^{d-1/2},
\end{split}    
\end{equation}
where $C_d(\beta)$ is a constant depending only on $\mathcal{W}$ and $\beta$.
Now the required statement follows by combining 
\eqref{equation for bound in cps 2}, \eqref{proof of lemma for
  primitive lemma error bound d equal 2}, and \eqref{proof of lemma
  for primitive lemma error bound d greater equal 3}.
%
\end{proof}

Finally we are ready for the
\begin{proof}[Proof of Theorem \ref{theorem on error term}]
By equation \eqref{eq: under H condition}, Lemma \ref{lemma of primitive points on cps},
and
for $M_T$ as in \eqref{eq: def MT}, we have
\begin{equation}\label{equation on upper bound on vis 1}
\begin{split}
\left|\#(\Lambda(\mathcal{W},\mathcal{L})_{\vis}\medcap TD)-
  \left(1-\frac{1}{\lambda^d}\right)\frac{M_T}{\zeta_{\mathcal{O}_K}(d)}\right|&\leq
\left|\#(\Lambda_{\textup{pr}}(\mathcal{W},\mathcal{L})\medcap TD)-
  \frac{M_T}{\zeta_{\mathcal{O}_K}(d)}\right|\\ 
&\,\, +
\left|\#\left(\Lambda_{\textup{pr}}\left(\frac{1}{\lambda}\mathcal{W},\mathcal{L}\right)\medcap
    TD\right)-
  \frac{1}{\lambda^d}\frac{M_T}{\zeta_{\mathcal{O}_K}(d)}\right|\\ 
&\leq \begin{cases}
 ( C_2(1) +C_2(1/\lambda)) T^{3/2}\log T,&\qquad d=2,\\
 (C_d(1)+ C_d(1/\lambda)) T^{d-1/2},&\qquad d\geq 3.
\end{cases}
\end{split}
\end{equation}
By \eqref{equation on upper bound on vis 1}, we have
\begin{equation}\label{equation on upper bound on vis 2}
\begin{split}
\left|\#(\Lambda(\mathcal{W},\mathcal{L})_{\vis}\medcap TD)-
  \left(1-\frac{1}{\lambda^d}\right)\frac{M_T}{\zeta_{\mathcal{O}_K}(d)}\right| 
&= \begin{cases}
 O(T^{3/2}\log T), &\qquad d=2,\\
  O (T^{d-1/2}),&\qquad d\geq 3.
\end{cases}
\end{split}
\end{equation}
\end{proof}

\bigskip
\end{document}